\documentclass[11pt,twoside, a4paper]{amsart}

\usepackage{mathptmx}
\usepackage{amssymb}
\usepackage{eucal}
\usepackage{amsxtra}
\usepackage{verbatim}
\usepackage{url}
\usepackage{enumerate}
\usepackage[english]{babel}
\usepackage[T1]{fontenc}
\usepackage{graphicx}
\usepackage{mathrsfs}
\usepackage{amscd,latexsym,amsthm,amsfonts,amssymb,amsmath,amsxtra}
\usepackage[colorlinks, urlcolor=blue,  citecolor=blue]{hyperref}
\usepackage{enumerate}
\usepackage[all]{xy}%
\usepackage{comment}
\theoremstyle{plain}

\newtheorem{theorem}{Theorem}
\newtheorem{proposition}[theorem]{Proposition}
\newtheorem{lemma}[theorem]{Lemma}
\newtheorem{corollary}[theorem]{Corollary}

\newtheorem{conditional-result}[theorem]{Conditional Result}

\newtheorem{theorem?}{Theorem(?)} [section]
\newtheorem{proposition?}[theorem]{Proposition(?)}
\newtheorem{lemma?}[theorem]{Lemma(?)}
\newtheorem{corollary?}[theorem]{Corollary(?)}

\newtheorem*{theorem*}{Theorem}
\newtheorem*{proposition*}{Proposition}
\newtheorem*{lemma*}{Lemma}
\newtheorem*{corollary*}{Corollary}
\newtheorem*{question*}{Question}
\newtheorem*{conjecture*}{Conjecture}
\newtheorem*{claim*}{Claim}

\newtheorem*{introtheorem*}{Theorem}
\newtheorem*{introproposition*}{Proposition}
\newtheorem*{introlemma*}{Lemma}
\newtheorem*{introcorollary*}{Corollary}

\theoremstyle{definition}

\newtheorem*{definition*}{Definition}
\newtheorem*{example*}{Example}

\theoremstyle{remark}
\newtheorem{remark}[theorem]{Remark}

\newtheorem*{remark*}{Remark}

\numberwithin{equation}{section}
\numberwithin{theorem}{section}

\DeclareSymbolFont{rsfs}{U}{rsfs}{m}{n}
\DeclareSymbolFontAlphabet{\mathcal}{rsfs}

\newcommand{\ZZ}{{\mathbb{Z}}}
\newcommand{\QQ}{{\mathbb{Q}}}
\newcommand\OO{\mathbb {O}}

\newcommand\YY{\mathfrak{Y}}
\newcommand\ttt{\mathbf{t}}
\newcommand\ww{\mathbf{w}}

\newcommand\zz{\mathbf{z}}

\newcommand{\Gal}{{\rm Gal}}
\newcommand{\Pic}{{\rm Pic}}
\newcommand{\Br}{{\rm Br}}

\newcommand{\Hom}{{\rm Hom}}

\newcommand{\kbar}{{\overline{k}}}

\newcommand{\Ybar}{{\overline Y}}

\newcommand{\Ubar}{{\overline U}}

\newcommand{\Deltabar}{{\overline \Delta}}
\newcommand{\Mbar}{{\overline M}}

\newcommand{\That}{{\widehat{T}}}

\newcommand{\Mhat}{{\widehat{M}}}

\newcommand{\Shat}{{\widehat{S}}}

\newcommand{\Res}{{\rm Res}}

\newcommand{\Xbar}{{\overline{X}}}

\def\G{{\mathbb{G}}}
\newcommand{\R}{{\textup{R}}}

\def\Z{{\ZZ}}
\def\Q{{\QQ}}

\newcommand{\Tbar}{{\overline T}}

\def\T{{\mathcal{T}}}
\def\A{\mathbf{A}}

\begin{document}	
\title[Strong approximation for a toric variety]
{           %{\protect\hfill \normalfont \tiny
	%	\\ \vspace{10pt}}
	Strong approximation for a toric variety
}

\author{Dasheng Wei}

%\address\{Borovoi:
%Raymond and Beverly Sackler School of Mathematical Sciences,
%Tel Aviv University, 69978 Tel Aviv, Israel\}
%\textbackslash\{\}email\{borovoi@post.tau.ac.il\}

%\author{Cyril Demarche}

%\address{Demarche:
%Universit\'e Pierre et Marie Curie (Paris 6),
%Institut de Math\'ematiques de Jussieu,
%4 place Jussieu, 75252 Paris Cedex 05,
%France}
%\email{demarche@math.jussieu.fr}

%\thanks{M. Borovoi a \'et\'e partiellement soutenu
%par le Centre Hermann Minkowski  pour la G\'eom\'etrie}

\date{\today}


\begin{abstract} Let $X$ be a toric variety  over a number field $k$ with $\kbar[X]^\times=\kbar^\times$. Let $W\subset X$ be a closed subset of codimension at least $2$. We prove that $X\setminus W$ satisfies strong approximation with algebraic Brauer--Manin obstruction.
\end{abstract}

\maketitle

\section{Introduction}
Let $X$ be a variety over a number field $k$. Let  $S$ be a finite set of places of $k$. One says that \emph{ strong approximation} holds for $X$ off $S$ if the diagnal image of the set $X(k)$ of rational points is dense in the $S$-adelic space
$X(\A^S_k)$ equipped with the adelic topology. Strong approximation for $X$ off $S$
implies a local-global principle for the existence of integral points on integral models of X over the ring of S-integers of $k$.

%For a proper variety $X$, we have $X(\A^S_k ) = \prod_{v \notin S} X(k_v)$, and the adelic topology coincides with the product
%topology. A proper variety satisfies strong approximation off $S$ if and only if weak
%approximation for the rational points holds off $S$. For a non-proper variety $X$, however, studying strong
%approximation seems to be generally much harder than studying weak approximation  for its proper models.

Strong approximation has been widely studied for algebraic groups and their homogeneous spaces. For a semisimple, almost simple, simply
connected linear algebraic group $G$ such that $\prod_{v\in S}G(k_v)$ is not compact,
strong approximation off $S$ was established by Eichler, Kneser, Shimura,
Platonov and Prasad. Strong approximation does not hold for groups which are not simply connected, but one may define a Brauer-Manin set. For homogeneous spaces of connected algebraic groups with connected stabilizers, see for example \cite{CTX09,Har08,Dem11,WX12,WX13,BD13} for strong approximation with Brauer--Manin obstruction.

For more general varieties which are not homogeneous
spaces, few strong approximation results are known. For strong approximation with Brauer--Manin obstruction for affine varieties defined by equations of the form
$$P(t) = q(z_1,z_2,z_3),$$
see \cite{CTX13}. For more general fibrations over $\mathbb A_k^1$
with split (e.g., geometrically integral) fibers, see \cite{CTH16}. For some affine varieties defined by equations of the form
$$P(\ttt) = N_{K/k}(\zz),$$
here $P(\ttt) \in k[t_1,...,t_s]$ is a polynomial over $k$ and $N_{K/k}$ is
a norm form for a field extension $K/k$, see \cite{DW18}.

All notation is standard, we refer it to \cite{CTX13}.
We study strong approximation with Brauer--Manin
obstruction for a toric variety. More precisely, we have the following result.

\begin{theorem*} \label{thm:toric} Let $k$ be a number field and $v_0$ is a place of $k$. Let $X$ be a smooth toric variety over $k$ with $\kbar[X]^\times=\kbar^\times$. Let $W\subset X$ be a closed subset of codimension at least $2$.
Then $X\setminus W$ satisfies strong approximation with algebraic Brauer--Manin obstruction off $v_0$.
\end{theorem*}

 A similar result to the theorem was also proved by Cao and Xu in \cite{CX18}. In fact, they proved any toric variety $X$ satisfies strong approximation with Brauer--Manin obstruction off all infinite places (without $\kbar[X]^\times=\kbar^\times$). However, if $\kbar[X]^\times\neq \kbar^\times$, generally $X\setminus W$ does not satisfy strong approximation with Brauer--Manin obstruction off infinite places (see \cite[Example 5.2]{CX18}). Therefore our result cannot be generalized to their case.

In Section \ref{section:application}, we give some examples which are defined by one multi-norm equation, the computation of Brauer groups is also given in this section.

\section{The proof of the main theorem} \label{section:main}

For a number field $k$, fix an algebraic closure $\kbar$,
and let $\Gamma_k$ be the absolute Galois group.
In this section, we mainly prove Theorem \ref{thm:toric} by descent theory (see \cite{CTS87}).

\begin{lemma} \label{lemma: SA-codimension2}

Let $Z$ be a closed subset of $\mathbb A^n_k$ of codimension at least $2$. Then $Y:=\mathbb A^n_k\setminus Z$ satisfies strong approximation off $v_0$, where $v_0$ is a place of $k$.
\end{lemma}
\begin{proof}
 Let $\mathfrak Y$ be an integral model of $Y$, $i.e.$, a separated $\OO_k$-scheme of finite type with the generic fiber $\YY\times_{\OO_k} k\simeq Y$. For any finite set $S \subset \Omega_k \setminus
  \{v_0\}$ containing $\infty_k \setminus \{v_0\}$ and any
  \begin{equation*}
    (p_v) \in \prod_{v \in S \cup \{v_0\}} Y(k_v) \times
    \left(\prod_{v \notin S \cup \{v_0\}} \YY(\OO_v)\right),
  \end{equation*}
  we must find $p \in Y(k)$ arbitrarily close to $p_v$ for all $v \in
  S$ with $p \in \YY(\OO_v)$ for all $v \notin S \cup \{v_0\}$.

We will prove the Lemma by replacing  $Y$ with a line ($\cong_k \mathbb A^1_k$) in $Y$. In fact, it is enough to find a line $\ell \subset
  Y$ such that
  \begin{enumerate}
  \item $\ell(k_v)$ contains a point $p_v'$ very close to $p_v$ for all $v
    \in S$, and
  \item $\ell(k_v) \cap \YY(\OO_v) \ne \emptyset$ for
    all $v \notin S \cup \{v_0\}$.
  \end{enumerate} Then the proof follows from the fact that $\ell \cong_k \mathbb A^1_k$ satisfies strong approximation off $v_0$.

  For this, we construct a suitable line $\ell$. We can choose a point $M \in Y(k)$
  arbitrarily close to $p_v$ for $v \in S \cup \{v_0\}$ because
  $\mathbb A^n_k$ satisfies weak approximation.  There is a finite
  set $S' \subset \Omega_k \setminus (S \cup \{v_0\})$ such that, for all
  $v \notin S' \cup S \cup \{v_0\}$, we have $M \in \YY(\OO_v)$. This gives (1),(2) outside $S'$ for any line through $M$.

  Recall $Z=\mathbb A^n_k\setminus Y$ has codimension at least $2$. The closure $Z'$ of the union of all
  lines through $M$ meeting $Z$ is a closed subvariety of
  $\mathbb A^n_k$ of codimension at least $1$. Therefore, $Y' := \mathbb A^n_k \setminus Z'$ is an
  open dense subset of $\mathbb A^n_k$. For all $v \in S'$, we choose
  arbitrary $p_v \in \YY(\OO_v)$. Then we have $N \in Y'(k)$ very
  close to $p_v$. This gives (2) for $S'$ for any line through $N$.
  Let $\ell$ be the line through $M$ and $N$, then $l\subset Y$ and satisfies (1) and (2), hence we complete the proof.
\end{proof}

A toric variety is an algebraic variety containing a torus as an open dense subset, such that the action of the torus on itself (by multiplication) extends to the whole variety. In this section, we will prove the main Theorem. The key idea of this proof is from \cite[Lemma 3.1]{DW18}.

\begin{proof}[Proof of the main theorem ]
%We may assume $(X\setminus W)(\A_k)^{\Br_1}\neq \emptyset$. Since Brauer--Manin obstruction is only one to the Hasse principle and weak approximation for toric varieties (e.g. \cite[Theorem 6.3.1]{Sko01}), then we may assume that $X$ contains an open subset which is isomorphic to the torus $T$.

Let $T$ be the torus contained in $X$ as an open dense subset and $\That$ its group of characters which is a $\Gamma_k$-module. Denote $$N:=X_*(T)=\Hom_\Z(\That,\Z)$$ which has a $\Gamma_k$-action such that $\langle \cdot\ ,\ \cdot \rangle : N\times \That \to \Z$ is the $\Gamma_k$-invariant bilinear pairing. Let $\Delta$ be a fan in $N_{\mathbb R}:=N\otimes \mathbb R$ which is $\Gamma_k$-invariant, $i.e.$, for any $\tau\in \Delta$ and $\sigma \in \Gamma_k$, we have $\sigma(\tau)\in \Delta$. Let $X=X_\Delta$  be the toric variety over $k$ associated to $\Delta$. Since $\kbar[X]^\times=\kbar^\times$,  we have $\Delta$ spans $N_\mathbb R$.
%First we assume  $\Delta$ can spans $N_\Bbb R$, hence $\kbar[X]^\times=\kbar^\times$.

Let $\text{Div}_{\Tbar}(\Xbar)$ be the group of $\Tbar$-invariant Weil divisors of $\Xbar$. By \cite[Theorem 4.1.3]{Cox}, we have the following exact sequence
\begin{equation}\label{equ:toric-1}
 0\to \That(\cong \kbar[T]^\times/\kbar^\times) \xrightarrow{\text{div}} \text{Div}_{\Tbar}(\Xbar) \to \Pic(\Xbar) \to 0.
\end{equation}
Let $f: \T\to X$ be a universal torsor of $X$. Let $\Mhat :=\text{Div}_{\Tbar}(\Xbar)$ and $M$ the dual torus of $\Mhat$. Let $i: M\to T$ be the induced morphism by $\text{div}: \That \to \Mhat$.  We now want to apply \cite[Theorem 2.3.1, Corollary 2.3.4]{CTS87} for the
local description of universal torsors of $X$.
The restriction $\T_T$ of the universal torsor $\T$ to $T$ has the form $\T_T  = M \times_T T$, which is the pull-back of $i: M\to T$ and $\sigma: T\to T$, where $\sigma$ is a translation of $T$ induced by a splitting of the following exact sequence
$$0\to \kbar^\times \to \kbar[T]^\times \to \kbar[T]^\times/\kbar^\times \to 0.$$
%In fact $\sigma$ is a translation of $T$.
We have the following commutative diagram:
$$
  \begin{CD}
    M \times_T T@>{p_1}>> M\\
    @V V {p_2}V @V V i V \\
    T @>{\sigma}>> T,
  \end{CD}
$$
where $p_1$ and $p_2$ are the natural projections. Since $\sigma$ is an isomorphism,
 the universal torsor $\T_T$ is isomorphic to $M$. Therefore, we may replace $\T_T$ by $M$ with the structure morphism $\sigma^{-1}\circ i: M\to T$.

Since $W$ has codimension at least $2$, then  $$\kbar[X\setminus W]^\times=\kbar[X]^\times=\kbar^\times \text { and } \Pic(\Xbar\setminus \overline{W})\cong\Pic(\Xbar).$$
Let $S$ be a group of multiplicative type.
We have the following commutative diagram (see for example \cite[Corollary 2.3.9]{Sko01}):
\begin{equation}\label{diag:torsor}
  \begin{CD}
    0 @>>> H^1(k,S)@>>> H^1(X,S) @>{\text{type}}>> \Hom(\Shat, \Pic(\Xbar))\\
    @. @| @V V \Res V @VV{\cong}V @.\\
    0 @>>> H^1(k,S)@>>> H^1(X\setminus W,S) @>{\text{type}}>> \Hom(\Shat, \Pic(\Xbar\setminus \overline W)).
  \end{CD}
\end{equation}
It is clear that universal torsors of $X$ exist, see \cite[Proposition 6.1.4]{Sko01}. Therefore, any universal torsor $\T'$ of $X\setminus W$ is the restriction of a universal torsor $f:\T \to X$ to $X\setminus W$, hence $\T'=\T\setminus f^{-1}(W)\subset \T$ and the restriction of $\T'$ to $T\setminus W$ is $M\setminus (i^{-1}\circ \sigma)(W)$.

Since $\Mhat$ is a permutation $\Gamma_k$-module, there are finite field extensions $K_1,\dots,K_n$ over $k$ such that  $\Mhat\cong \prod_{i=1}^{n}\Z[K_i/k]$.
Therefore
\begin{equation*}%\label{}
M\cong \prod_{i=1}^{n} \R_{K_i/k}(\G_{m,K_i}).
\end{equation*}

Denote $Y:=\prod_i \R_{K_i/k}(\A^1_{K_i})$, obviously $Y$ is an affine space over $k$. We have the natural embedding $ M\subset  Y.$  We claim:

\emph{There is a closed subvariety $Z\subset Y$ of codimension at least $2$ and $Y\setminus Z \supset M\setminus (i^{-1}\circ \sigma)(W)$, such that the restriction of $\sigma^{-1} \circ i$ to $M\setminus (i^{-1}\circ \sigma)(W)$ %$\rightarrow T\setminus W$
can be extended to $ Y\setminus Z\rightarrow X\setminus W$.} \newline
This claim implies our theorem as follows. By Lemma \ref{lemma: SA-codimension2}, $Y\setminus Z$ satisfies strong approximation off $v_0$. Therefore,  by \cite[Lemma 3.1]{DW18}, $X\setminus W$ satisfies strong approximation with algebraic Brauer--Manin obstruction off $v_0$.

We now proves this claim. Since $X$ is a toric variety which has a $T$-action, the translation $\sigma: T\to T$ can be extended to a unique isomorphism of $X$ which we also denote by $\sigma$. Obviously $W':=\sigma^{-1}(W)\subset X$ also has codimension at least $2$. Therefore we only need to show that there exists $Z\subset Y$ such that $i: M\setminus i^{-1}(W')\to T\setminus W'$ can be extended to $Y\setminus Z\rightarrow X\setminus W'$, noting that the composite map $Y\setminus Z \to X\setminus W'\xrightarrow{\sigma} X\setminus W$ gives the map we are looking for.

We only need to  show such $Z$ exists over $\kbar$. Indeed, if we have found such $g:\overline{Y}\setminus {Z}\rightarrow \Xbar\setminus \overline {W'}$ over $\kbar$, then we also get the twisted map $g^\tau: \overline{Y}\setminus \tau(Z)\rightarrow \Xbar\setminus \overline {W'}$ for any $\tau\in \Gamma_k$.  For $\tau_1,\tau_2\in \Gamma_k$, the two maps $g^{\tau_1}$ and $g^{\tau_2}$ restricted to their intersection $$(\overline{Y}\setminus\tau_1(Z))\cap (\overline{Y'}\setminus\tau_2(Z))=\overline{Y}\setminus(\tau_1(Z)\cup \tau_2(Z))$$ are the same by the reduced and separated property of $\overline{Y}$ and $\Xbar\setminus \overline {W'}$ over $\kbar$. Then we can replace $Z$ by the union  $\cup_{\tau\in \Gamma_k} \tau(Z)$ which is in fact a finite union and also has codimension at least $2$. Therefore we get a morphism $Y\setminus Z\rightarrow X\setminus \overline {W'}$ over $k$ and $Z$ has codimension at least $2$.

Now we assume $k=\kbar$. Let $\Deltabar$ be the fan of $\Delta$ omitting the $\Gamma_k$-action.

We shall call a 1-dimensional cone a ray. Let $\Deltabar(1)$ be the set of rays of $\Deltabar$. Let $D_\rho$ be the $\Tbar$-invariant Weil divisor associated to $\rho\in \Deltabar(1)$ (see \cite[Section 4.1]{Cox}). Recall $\Mhat=\bigoplus_{\rho\in \Deltabar(1)}\Z D_\rho$. Denote $\tilde N:=X_*(M)=\Hom_\Z(\Mhat,\Z)$, ${\tilde D}_\rho\in \tilde N$ the dual of $D_\rho$, and $\tilde N_\mathbb R:=\tilde N\otimes_\Z \mathbb R$. It is clear that $\{{\tilde D}_\rho: \rho\in \Deltabar(1)\}$ is a basis of $\tilde N_\mathbb R$.
%The morphism $\text{div}: \That\to \Mhat $ gives a morphism $f: \tilde N\to N$.

Let $C$ be the fan associated to the simplicial cone in $\tilde N_\mathbb R$ generated by $\{{\tilde D}_\rho:\rho\in \Deltabar(1)\}$. Then $\Ybar\cong Y_C$, the toric variety of $C$ with the natural $\Mbar$-action which is an affine space. Let $R_\rho\subset \tilde N_\mathbb R$ be the ray generated by ${\tilde D}_\rho$. Then $\{R_\rho: \rho\in \Deltabar(1)\}$  are all rays of $C$. Let
$$C':=\{0\}\cup \{R_\rho: \rho\in \Deltabar(1)\}\subset C,$$
which is a subfan of $C$. Denote by $Y_{C'}\ (\subset Y_C)$ the open toric subvariety of the subfan $C'$. Similarly, let $$\overline {\Delta'}:=\{0\}\cup \overline \Delta(1)\subset C.$$ Denote by $X_{\overline {\Delta'}}\ (\subset X_{\overline \Delta})$ the open toric subvariety of the subfan $\overline {\Delta'}$.

We will prove the claim as follows. We will show such $Z$ exists for $X$ ($i.e.$, $W'=\emptyset$), then we get a morphism $g: Y\setminus Z \to X$ which extends $i: M\to T$; in the following we will show $g^{-1}(W')\subset Y\setminus Z$ has codimension at least $2$. Replacing $Z$ by $Z\cup g^{-1}(W'))$, the proof of the claim as follows.

We will prove that such $Z$ exists for $X$. In fact, we will prove:

 \emph{There is a toric morphism $g: Y_{C'} \to X_{\overline {\Delta'}}$ which extends $\Mbar\to \Tbar$.}\newline
  Note that $Y_{C'}\subset Y_C$ and $Y_C\setminus Y_{C'}$ has codimension $2$ since $C'$ contains all rays in $C$.

Let $\tilde g: \tilde N_\mathbb R\to N_\mathbb R$ be the morphism induced by $\text{div}: \That\to \Mhat $.
The claim follows from the fact that there is a morphism of fans $C'\to \overline {\Delta'}$ which is induced by $\tilde g:\tilde N_\mathbb R\to N_\mathbb R$ (see \cite[Theorem 3.3.4]{Cox}).  %In the following we gives the precise description of $f$.

For any $\rho\in \Deltabar(1)$, let $u_\rho\in N$ be the minimal generator of $\rho$. By \cite[Proposition 4.1.2]{Cox}, we have $$\text{div}: \That \to \Mhat, \chi \mapsto \text{div}(\chi)=\sum_{\rho\in \Deltabar}\langle u_\rho,\chi\rangle D_\rho.$$
%The $f: \tilde N\to N$ is induced map $\text{div}: \That\to \Mhat $. Let $R_\rho$ be the ray generated by ${\tilde D}_\rho$.
%$\{{\tilde D}_\rho: \rho\in \Deltabar(1)\}$ is a basis of $\tilde N (=M^\vee)$.
For $\sigma\in \Deltabar(1)$, $\tilde g({\tilde D}_\sigma)\in N$ and
\begin{equation*}
\tilde g({\tilde D}_\sigma)=\left(\chi\mapsto {\tilde D}_\sigma(\sum_{\rho\in \Deltabar(1)}\langle u_\rho,\chi\rangle D_\rho)=\langle u_\sigma,\chi\rangle \right),
\end{equation*}
then $\tilde g({\tilde D}_\sigma)=u_\sigma$.
%Since the natural isomorphism $N\cong N^{**}$ is given by $x\mapsto (\chi \mapsto \chi(x))$, we have $f({\tilde D}_\sigma)=u_\sigma$.
% $f: \tilde N\to N$ is given by $\tilde D_\rho \mapsto u_\rho$.
Therefore $\tilde g(R_\rho)=\rho$ for any $\rho\in \Deltabar(1)$, hence $\tilde g$ induces a morphism of fans $C' \to \overline {\Delta'}$. 

%Let $\overline \Delta'=\{0\}\cup \overline\Delta(1)$ be the subfan of $\Delta$. Then we get a toric morphism $g: Y_{C'} \to X_{\Delta'}\subset X_{\overline \Delta}$

In the following, We will prove that $g^{-1}(W')\subset Y_{C'}$ is a closed subset of codimension $\geq2$.

%Since $\sigma:X \to X $ is an isomorphism , then $\sigma_X^{-1}(W)\subset X$ also has codimension $2$. Therefore we only need to show $g^{-1}(W)\subset (Y\setminus Z)$ has codimension 2.
Let $U_{R_\rho}$ (resp. $V_\rho$) be the toric open subvariety of $Y_{C'}$ (resp. $X_{\Delta'}$) associated to $R_\rho\in C'$ (resp. $\rho\in \overline \Delta(1)$). Let $g_\rho: U_{R_\rho}\to V_\rho$ be the induced toric morphism by $\tilde g:\tilde N\to N$ since $\tilde g(R_\rho)=\rho$. We have $Y_{C'}$ (resp. $X_{\Delta'}$) is the union of $U_{R_\rho}$ (resp. $V_\rho$), $\rho\in \overline\Delta(1)$. The map $g_\rho$ is the restriction of $g$ to $U_\rho$. Therefore, we only need to show that $g_\rho^{-1}(W'\cap U_{R_\rho})\subset U_{R_\rho}$ has codimension at least 2.

%Let $h: \text{Spec }\kbar[\Mhat]\to \text{Spec }\kbar[\That]$ is induced by $\text{div}: \That \to \Mhat$. We have the two natural embeddings as toric varieties $$Y_{C'}\subset \text{Spec }\kbar[\Mhat] \text{ and } X_{\overline {\Delta'}} \subset \text{Spec }\kbar[\That].$$ The morphism $g$ is the restriction of $h$ to $Y_{C'}$.
%We will show the inverse image of any subset of codimension at least $2$ also has codimension at least $2$.
We may write $\Mhat \cong \Mhat_1\oplus \Mhat_2$, where $\Mhat_1$ and $\Mhat_2$ free abelian groups, such that $\That \xrightarrow{\text{div}}\Mhat$ is identity with the composite map
\begin{equation}\label{morphism:d}
\That \xrightarrow{d} \Mhat_1\xrightarrow{j} \Mhat_1\oplus \Mhat_2,
\end{equation}
where the cokernel of $d$ is finite and $j$ is the natural embedding. Let $N':=X_*(M_1)=\Hom(\Mhat_1, \Z)$ be the dual of $\Mhat_1$, obviously $\tilde N\cong N'\oplus X_*(M_2)$, where $X_*(M_2)=\Hom(\Mhat_2, \Z)$. Therefore we have the dual morphism $$\tilde N \cong N'\oplus X_*(M_2) \xrightarrow{j^*} N' \xrightarrow{d^*} N,$$ where $j^*$ is in fact the natural projection. Let $R_\rho'\subset N'_\mathbb R:=N'\otimes_\Z \mathbb R$ be the ray of the image of $R_\rho$ by $j^*\otimes \mathbb R$. Let $U'_{R_\rho'}$ be the toric variety associated to $R_\rho'\subset N'_\mathbb R$. Obviously $U_{R_\rho}\cong U'_{R_\rho'}\times \G_m^m$ where $m=\dim(\Mhat_2)$, and $d^*$ induces a toric morphism $g_\rho': U'_{R_\rho'} \to V_\rho$.  Then $g_\rho$ is identity with the composite map $$U'_{R_\rho'}\times \G_m^m \xrightarrow{p_1} U'_{R_\rho'} \xrightarrow{g_\rho'}  V_\rho,$$ where $p_1$ is the natural projection.

Let $\rho^\vee\subset \That\otimes_\Z \mathbb R$ (resp. $R_\rho'^\vee\subset \Mhat_1 \otimes_\Z \mathbb R$) be the dual cone of $\rho$ (resp. $R_\rho'$). Then the morphism $d$ in (\ref{morphism:d}) induces a morphism $d_\rho: k[\rho^\vee\cap \That] \to k[R_\rho'^\vee\cap \Mhat_1]$. Since the cokernel of $d$ is finite, there exists $m>0$ such that $m\rho'\in d(\rho^\vee\cap \That)$ for any $\rho'\in R_\rho'^\vee \cap \Mhat_1$. Therefore, $d_\rho$ is a finite and injective morphism, $i.e.$, $g_\rho'$ is  a finite and surjective morphism.
\end{proof}
\begin{remark}
\begin{enumerate}
\item Let $N$ be a lattice and $N_\mathbb R=N\otimes_\Z \mathbb R$. Let $\Delta$ be a fan in $N_\mathbb R$ and $X=X_\Delta$ the toric variety of  $\Delta$. In fact, the condition $\kbar[X]^\times =\kbar^\times$ is equivalent to  the fact that $\Delta$ spans $N_\mathbb R$ over $\mathbb R$.

\item If $W$ is not empty, the condition $\kbar[X]^\times=\kbar^\times$ is necessary. For example, let $k=\Q$, $X=\G_a\times \G_m$ a toric variety with $\kbar[X]^\times/\kbar^\times\cong \Z$, $W=\{(0,1)\}\subset X$. Then $X\setminus W$ does not satisfy strong approximation with Brauer--Manin obstruction off the infinite place (see \cite[Example 5.2]{CX18}).
\end{enumerate}
\end{remark}

\section{Some examples} \label{section:application}

In this section, we will give some typical varieties which satisfy strong approximation with algebraic Brauer--Manin obstruction off $v_0$. In fact this is a direct consequence of the main theorem.
\begin{proposition} \label{prop: multinorm} Let $m,n \geq 1$ and let $K_i,L_j$ be finite field extensions over a number field $k$, where $i=1,\dots,m$ and $j=1,\dots,n$. Let $X$ be the smooth locus of the affine variety defined by
\begin{equation}\label{equ:multi-1} \prod_{i=1}^m N_{K_i/k}(\ww_i)=c\prod_{j=1}^n  N_{L_j/k}(\zz_j)^{s_j},
\end{equation}
with $c\in k^\times$ and $s_j\geq 1$, $j=1,\dots, n$. Then $\kbar[X]^\times=\kbar^\times$ and $X$ satisfies strong approximation with algebraic Brauer--Manin obstruction off $v_0$.
\end{proposition}
\begin{proof} It is clear that $X$ contains an open subset $U$ by $\prod_{i=1}^m N_{K_i/k}(\ww_i)\neq 0$, which is a principal homogeneous space of the torus defined by $$\prod_{i=1}^m  N_{K_i/k}(\ww_i')\cdot\prod_{j=1}^n N_{L_j/k}(\zz_j')^{s_j}=1.$$
In fact, $X$ is a toric variety.

Over $\kbar$, $X$ can be viewed as the smooth locus of the affine variety defined by
\begin{equation}\label{equ:multi-2}
w_1 \cdots w_{m'}=cz_1^{r_1}\cdots z_{n'}^{r_{n'}}
\end{equation}
with all $r_j\geq 1$. Let $D_{i,j}$ be the prime divisor of $X$ defined by $w_i=z_j=0$ for $i,j$.
We can see $\text{div}(w_i)=\sum_{j=1}^{n'}r_j D_{ij}$ and $\text{div}(z_j)=\sum_{i=1}^{m'}D_{ij}$.

We now show $\kbar[X]^\times=\kbar^\times$.
Since every $f \in
  \kbar[X]^\times$ has the form
  \begin{equation*}
    f = c'w_1^{\alpha_1}\cdots w_{m'}^{\alpha_{m'}} z_1^{\beta_1}\cdots z_{n'}^{\beta_{n'}}
  \end{equation*}
  with $c' \in \kbar^\times$ and $\alpha_1, \dots, \alpha_{m'}, \beta_1, \dots, \beta_{n'} \in
  \ZZ$, and
  \begin{equation*}
    0 = \textup{div}(f)=\sum_{i=1}^{m'} \sum_{j=1}^{n'} (\alpha_i r_j+\beta_j) D_{i,j},
  \end{equation*}
  and $D_{i,j}$ are linearly independent, we have $\beta_j=-\alpha_i r_j$ for $i,j$,  therefore $\alpha_1=\cdots=\alpha_{m'}$ and $\beta_j=-\alpha_1 r_j$ for all $j$, then $f$ is a constant by (\ref{equ:multi-2}).

  Since $X$ is a toric variety, the proof follows the main theorem, but here we give a direct argument. For any universal torsor $\T$ of $X$, the restriction $\T_U$ is a quasi-split torus, and $\T_U\cong \prod_{i,j} \R_{K_i\otimes_k L_j/k}(\G_m)$,  and the map $\T_U \to U$ is given by $$(u_{ij})_{i,j} \mapsto (\eta_i\prod _{j=1}^n  N_{K_i\otimes_k L_j/K_i}(u_{ij})^{s_j})_i\times (\xi_j \prod _{i=1}^m N_{K_i\otimes_k L_j/L_j}(u_{ij}))_j$$ with $\eta_i \in K_i^\times,\xi_j\in L_j^\times$ such that $\prod _iN_{K_i/k}(\eta_i)=c \prod _j N_{L_j/k}(\xi_j)^{s_j}$. Obviously $\T_U$ can naturally embed into $$Y:=\prod_{i,j} \R_{K_i\otimes_k L_j/k}(\mathbb A^1)\cong \mathbb A_k^m$$ for some $m$, hence $\T_U \to X$ can be extended to $g: Y\to X'$, where $X'$ is the affine variety defined by (\ref{equ:multi-1}).
  By the equation (\ref{equ:multi-2}), the singular locus $V$ of $X'$, $i.e.$, $X'\setminus X$ is contained in $W$ which has codimension 2, where $W$ is the union of
  $$\{w_i=w_j=0, z_l=0\} \text{ with } i\neq j.$$ Therefore, the map $\T_U \to U$ can be extended to $Y\setminus f^{-1}(V)\to X$. Since $f^{-1}(V)\subset f^{-1}(W)$, then $f^{-1}(V) \subset Y$ has codimension at least $2$. By Lemma \ref{lemma: SA-codimension2}, $Y\setminus f^{-1}(V)$ satisfies strong approximation off $v_0$. By \cite[Lemma 3.1]{DW18}, $X$ satisfies strong approximation with algebraic Brauer--Manin obstruction off $v_0$.
\end{proof}

\begin{corollary} \label{lem:Br-1} Let $L$ and $K$ be finite Galois extensions over $k$ . Let $X$ be the smooth locus of the affine variety over $k$ defined by
$$ N_{L/k}(\ww)=cN_{K/k}(\zz)$$
with $c\in k^\times$, and let $T$ be the multi-norm $1$ torus over $k$ defined by $ N_{L/k}(\ww')N_{K/k}(\zz')=1$. Then
\begin{equation*} \Br_1(X)/\Br_0(X)\cong H^2(L.K/k, \That).
\end{equation*}
Furthermore, if $L\cap K=k$ and the degrees of the maximal abelian subextensions of $L/k$ and $K/k$ are relatively  prime, then $\Br_1(X)=\Br_0(X)$ and $X$ satisfies strong approximation off $v_0$.
\end{corollary}
\begin{proof} By Proposition \ref{prop: multinorm}, $\kbar[X]^\times=\kbar^\times$. Let $U$ be the open subset of $X$ by $N_{L/k}(\ww)=cN_{K/k}(\zz) \neq 0$, then $U$ is a principal homogeneous space of $T$ and $\kbar[U]^\times/\kbar^\times \cong \That$ as $\Gamma_k$-modules, hence we have the following exact sequence
\begin{equation}\label{seq:picard-1}
0\rightarrow \That \rightarrow \textup{Div}_{\Xbar\setminus\Ubar}(\Xbar)\to \Pic(\Xbar) \to \Pic(\Ubar)=0.
\end{equation}
Let $G_1:=\Gal(L/k)$, $G_2:=\Gal(K/k)$ and $G:=\Gal(L.K/k)$.
Since $\textup{Div}_{\Xbar\setminus\Ubar}(\Xbar)\cong \Z[G_1]\bigotimes \Z[G_2]$ as $\Gamma_k$-module and split by $L.K$, then we can view $\textup{Div}_{\Xbar\setminus\Ubar}(\Xbar)$ as a $G$-module. For $i\geq 1$, by Shapiro's Lemma, we have $$H^i(G,\textup{Div}_{\Xbar\setminus\Ubar}(\Xbar))=H^i(L.K/K,\Z[G_1])=H^i(L/L\cap K, \Z[G_1])=0.$$

The long exact sequence in Galois cohomology associated to (\ref{seq:picard-1}) gives the exact sequence
\begin{equation*}%\label{seq:H1(picard)-1}
H^1(G,\textup{Div}_{\Xbar\setminus\Ubar}(\Xbar))\to H^1(G,\Pic(X_{L.K})) \to H^2(G,\That)\to H^2(G,\textup{Div}_{\Xbar\setminus\Ubar}(\Xbar)),
\end{equation*}
which implies $$ H^1(G,\Pic(X_{L.K})) \cong H^2(G,\That).$$
Since $\Pic(\Xbar)$ is free and split by $L.K$, we have $$\Br_1(X)/\Br_0(X)\cong H^1(k,\Pic(\Xbar))\cong H^1(G,\Pic(X_{L.K})) \cong H^2(G,\That).$$

The long exact sequence in Galois cohomology associated to the exact sequence $$0\rightarrow \Z \to \Z[G_1]\oplus \Z[G_2]\to \That\to 0$$
gives the exact sequence \begin{equation}\label{seq:That}
\aligned H^2(G,\Z)&\xrightarrow{i} H^2(L.K/L,\Z)\oplus H^2(L.K/K,\Z)\to H^2(G,\That)\\
&\to H^3(G,\Z)\xrightarrow{j} H^3(L.K/L,\Z)\times H^3(L.K/K,\Z),
\endaligned \end{equation}
where the maps $i$ and $j$ are induced by the restriction maps. If $L\cap K=k$, we have $\Gal(L.K/K)\cong G_1$, $\Gal(L.K/L)\cong G_2 $ and $G\cong G_1\times G_2$ canonically, it is clear that $i$ is surjective. On the other hand, $$H^1(G_1,H^1(G_2,\Q/\Z))=\Hom(G_1^{ab},\Hom(G_2^{ab},\Q/\Z)),$$
where $G_1^{ab}$ and $G_2^{ab}$ are their maximal abelian quotient. By our assumption, $\# G_1^{ab}$ and $\# G_2^{ab}$ are relatively prime, hence $$H^1(G_1,H^1(G_2,\Q/\Z))=0.$$
By the K\"{u}nneth formula (\cite[Exercise II.1.7]{NSW99}), the  map $j$ in (\ref{seq:That}) is injective, hence (\ref{seq:That}) implies $H^2(G,\That)=0$. Therefore $$\Br_1(X)=\Br_0(X).$$ By Proposition \ref{prop: multinorm}, $X$ satisfies strong approximation off $v_0$.
\end{proof}

\bf{Acknowledgment} 	 %Daniel Loughran pointed out a mistake
 \it{The work is supported by 
  National Natural Science Foundation of China (Grant Nos. 11622111, 11631009, 11621061 and 11688101). The author thanks Yifei Chen and D. Loughran for useful comments and discussions.}

%%%%%%%%%%%%%%%%%%%%%%%%%%%%%%%%%%%%%%%%%%%%%%%%%%%%%%%%%%%%%%%%%%%%%%%%%%%%%%%%%%%%%

%%%%%%%%%%%%%%%%%%%%%%%%%%%%%%%%%%%%%%%%%%%%%%%%%%%%%%%%%%%%%%%

\bigskip
{\small

{\scshape
D. Wei: Academy of Mathematics and System Science,  CAS, Beijing
100190, P.R.China and School of mathematical Sciences, University of  CAS, Beijing
100049, P.R.China
}
\smallskip

{\it E-mail: }
\url{dshwei@amss.ac.cn}
}

\end{document}